\documentclass[11pt]{amsart}

\usepackage[T1]{fontenc}
\usepackage[utf8]{inputenc}
\usepackage[english]{babel}
\usepackage{lmodern}
\usepackage[margin=1in]{geometry}
\usepackage{amsmath,amssymb,amsfonts,amsthm,mathtools}
\usepackage{mathrsfs}
\usepackage{enumitem}
\usepackage{hyperref}

\hypersetup{
  colorlinks=true,
  linkcolor=blue,
  citecolor=blue,
  urlcolor=blue
}

\numberwithin{equation}{section}

\newtheorem{theorem}{Theorem}[section]
\newtheorem{proposition}[theorem]{Proposition}
\newtheorem{lemma}[theorem]{Lemma}
\newtheorem{corollary}[theorem]{Corollary}

\theoremstyle{definition}
\newtheorem{definition}[theorem]{Definition}

\newcommand{\R}{\mathbb{R}}
\newcommand{\Z}{\mathbb{Z}}
\newcommand{\N}{\mathbb{N}}
\newcommand{\cK}{\mathcal{K}}
\newcommand{\cP}{\mathcal{P}}
\newcommand{\cG}{\mathcal{G}}
\newcommand{\eps}{\varepsilon}
\newcommand{\weak}{\rightharpoonup}
\newcommand{\crit}{\operatorname{crit}}

\newcommand{\vertiii}[1]{{\left\vert\kern-0.25ex\left\vert\kern-0.25ex\left\vert #1 \right\vert\kern-0.25ex\right\vert\kern-0.25ex\right\vert}}

\title[Strongly indefinite Schrödinger equations]
{Ground states for strongly indefinite Schr\"{o}dinger equations with competing nonlinearities}

\author[B. Bieganowski]{Bartosz Bieganowski}

\address[B. Bieganowski]{\newline\indent
	Faculty of Mathematics, Informatics and Mechanics, \newline\indent
	University of Warsaw, \newline\indent
	ul. Banacha 2, 02-097 Warsaw, Poland}
\email{\href{mailto:bartoszb@mimuw.edu.pl}{bartoszb@mimuw.edu.pl}}

\date{}

\begin{document}

\begin{abstract}
We survey recent variational methods for strongly indefinite Schrödinger
equations with sign-changing nonlinearities. The main object is an energy
functional of the form
\[
J(u)=\frac12\|u^+\|^2-\frac12\|u^-\|^2
-\int_{\R^N}F(u)\,dx+\lambda\int_{\R^N}G(u)\,dx,
\]
where the splitting $X=X^+\oplus X^-$ is induced by a spectral gap of the
linear Schrödinger operator, and where the nonlinear part
\[
I(u)=\int_{\R^N}F(u)\,dx-\lambda\int_{\R^N}G(u)\,dx
\]
is allowed to change sign. We discuss the generalized linking theorem
developed for such functionals, and the abstract multiplicity theory for critical orbits in dislocation spaces. In the final part, we prove a new ground state result for the competing pure-power case
\[
f(u)=|u|^{p-2}u,\qquad g(u)=|u|^{q-2}u,\qquad 2<q<p<2^*.
\]
More precisely, for $\lambda>0$ sufficiently small, the corresponding
strongly indefinite functional possesses a nontrivial critical point of
least energy among all nontrivial critical points. 

\medskip

   

\end{abstract}

\maketitle

\section{Introduction}

We consider the nonlinear stationary Schr\"odinger equation
\begin{equation}\label{eq:model-basic}
-\Delta u + V(x)u = f(u)-\lambda g(u),
\qquad x\in \mathbb{R}^N,
\end{equation}
where $N \geq 3$, $\lambda\geq 0$, $V:\mathbb{R}^N\to\mathbb{R}$ is a periodic
potential, and $f,g:\mathbb{R}\to\mathbb{R}$ are nonlinearities of
subcritical growth.

Equations of the form \eqref{eq:model-basic} arise naturally in nonlinear
optics (\cite{Kuchment, Pankov}). Indeed, consider the time-dependent nonlinear Schr\"odinger
equation
\begin{equation}
\mathrm{i}\partial_t\Psi
=
-\Delta\Psi+V_0(x)\Psi-h(|\Psi|)\Psi,
\qquad (t,x)\in \mathbb{R}\times\mathbb{R}^N,
\end{equation}
where the nonlinear term $h$ describes the nonlinear polarization of the medium. Searching for standing
waves of the form
\[
\Psi(t,x)=u(x)e^{-\mathrm{i}\omega t},
\]
we obtain
\[
-\Delta u+\bigl(V_0(x)-\omega\bigr)u=h(|u|)u.
\]
Thus, after incorporating the frequency shift $\omega$ into the potential,
one is led to an equation of the form \eqref{eq:model-basic}. When the
potential is periodic, localized standing waves are usually referred to as
\emph{gap solitons} (\cite{Buryak}).

The main purpose of this paper is to survey recent variational results for
strongly indefinite Schr\"odinger equations with competing nonlinearities and to provide a new ground-state result.
More specifically, we summarize the results obtained in
\cite{BB2022,BBS2025}, where an existence theory and a multiplicity theory
were developed for strongly indefinite functionals whose nonlinear part may
change sign. The new contribution concerns the existence of a ground state in the
pure-power case. More precisely, for sufficiently small $\lambda>0$, we
prove the existence of a least-energy nontrivial solution for the model
problem
\begin{equation}\label{eq:pure-power-problem}
-\Delta u+V(x)u
=
|u|^{p-2}u-\lambda |u|^{q-2}u,
\qquad x\in\mathbb{R}^N,
\end{equation}
where
\[
2<q<p<2^*,
\]
with
\[
2^*=
\dfrac{2N}{N-2}.
\]
Here the term $|u|^{p-2}u$ is focusing, whereas
$-\lambda |u|^{q-2}u$ is defocusing. From the physical point of view,
this corresponds to the competition of two different mechanisms of
nonlinear polarization (\cite{Slusher}).

Throughout the paper, we assume that $V$ is $\mathbb{Z}^N$-periodic and
that $0$ lies in a spectral gap of the self-adjoint Schr\"odinger operator
\[
L:=-\Delta+V
\]
on $L^2(\mathbb{R}^N)$, see \cite{ReedSimon} for the full description of the spectrum of the periodic Schr\"odinger operator. Consequently, the natural energy space admits an
orthogonal spectral decomposition
\[
X=X^+\oplus X^-,
\]
where $X^+$ and $X^-$ are the positive and negative spectral subspaces
associated with $L$, respectively. We are interested in the genuinely
strongly indefinite case in which both $X^+$ and $X^-$ are
infinite-dimensional.

Writing $u=u^++u^-$, where $u^\pm\in X^\pm$, the corresponding energy
functional takes the form
\begin{equation}
J(u)
=
\frac{1}{2}\|u^+\|^2
-
\frac{1}{2}\|u^-\|^2
-
I(u),
\end{equation}
where
\[
I(u)
=
\int_{\mathbb{R}^N}F(u)\,dx
-
\lambda\int_{\mathbb{R}^N}G(u)\,dx,
\]
and
\[
F(t):=\int_0^t f(s)\,ds,
\qquad
G(t):=\int_0^t g(s)\,ds.
\]
Since the quadratic part of $J$ is positive on $X^+$ and negative on
$X^-$, the functional is unbounded both from above and from below on
infinite-dimensional subspaces. Hence neither direct minimization nor the
classical mountain-pass theorem \cite{AmbrosettiRabinowitz} can be applied directly.

Among the standard variational tools for strongly indefinite problems are
the generalized linking theorem of Kryszewski and Szulkin \cite{KS}, the
Nehari--Pankov method \cite{Pankov}, and the approach of Szulkin and Weth
\cite{SW}. In many classical applications, the nonlinear part $I$ is
assumed to be nonnegative. This assumption is technically convenient, it
yields suitable upper semicontinuity properties with respect to
product-type weak topologies and facilitates the comparison of minimax
levels with minimization levels on the Nehari--Pankov manifold. In this case, the existence of ground states and multiplicity of solutions have been studied by many authors; see, for example, \cite{Li,Liu,Mederski2016,Rabinowitz,dePaiva} and the references therein.

The situation studied here is more delicate. Although $F\geq 0$ and
$G\geq 0$, their difference may change sign. For the pure-power problem
\eqref{eq:pure-power-problem}, one has
\[
F(t)-\lambda G(t)
=
\frac{|t|^p}{p}
-
\lambda\frac{|t|^q}{q},
\]
which is negative for sufficiently small nonzero $|t|$ and positive for
large $|t|$. Therefore, the usual strongly indefinite variational
machinery cannot be applied without substantial modifications.

In \cite{BB2022}, Bernini and Bieganowski established a generalized
linking theorem for strongly indefinite functionals with a sign-changing
nonlinear part. In particular, they obtained a nontrivial critical point
for sufficiently small $\lambda>0$. Subsequently, in \cite{BBS2025},
Bernini, Bieganowski, and Strzelecki developed an abstract multiplicity
theory for critical orbits in this setting, based on dislocation spaces and
discrete group actions.

The new ground-state result proved in the final part of this paper is not
an immediate consequence of the existence theorem from \cite{BB2022}. The
critical point obtained there is produced by a linking argument and is not
identified as a least-energy solution. Moreover, the direct comparison with
the Nehari--Pankov manifold is available only in the limiting case
$\lambda=0$.

For the pure-power problem, we show that this difficulty can be overcome by adapting the minimization-on-critical-points strategy of Jeanjean and Tanaka \cite{JeanjeanTanaka}.
The proof combines a positivity estimate for bounded Palais--Smale
sequences whose norm stays away from zero with a
concentration--compactness argument and a Brezis--Lieb type splitting
procedure. This yields compactness modulo lattice translations and allows
us to identify a least-energy nontrivial critical point of $J$.

We impose the following assumption on the potential $V$.

\begin{enumerate} [label=(V),ref=V]
\item \label{ass:V}
The potential $V\in L^\infty(\R^N)$ is  $\Z^{N}$-periodic,
\[
0\notin\sigma\left(-\Delta+V(x)\right),
\qquad
\sigma\left(-\Delta+V(x)\right)\cap(-\infty,0)\ne\emptyset.
\]
\end{enumerate}

The natural energy space is
\[
X:=H^1(\R^N).
\]
Throughout, $|\cdot|_r$ denotes the norm of $L^r(\R^N)$, and
$a\lesssim b$ means that $a\le Cb$ for a positive constant $C$ independent
of the quantities under consideration. Under (\ref{ass:V}), the quadratic form
\[
Q(u)=\int_{\R^N}\left(|\nabla u|^2+V(x)|u|^2\right)\,dx
\]
has a spectral decomposition into positive and negative parts. More precisely, one
obtains an orthogonal splitting
\[
X=X^+\oplus X^-
\]
such that $Q$ is positive (respectively, negative) definite on $X^+$ (respectively, $X^-$). For $u^\pm\in X^\pm$ we define
\[
\|u^\pm\|^2
=
\pm\int_{\R^N}\left(|\nabla u^\pm|^2
+V(x)|u^\pm|^2\right)\,dx,
\]
and
\[
\|u\|^2=\|u^+\|^2+\|u^-\|^2.
\]
We denote by $\langle\cdot,\cdot\rangle$ the scalar product associated
with this norm. The presence of the spectral gap around zero implies the existence of a constant $\mu_0>0$ such that
\begin{equation}\label{eq:mu0}
\mu_0|u|_2^2\le \|u\|^2,\qquad u\in X.
\end{equation}
Moreover, the projections $u\mapsto u^\pm$ are continuous in
$L^s(\R^N)$ for every $2\le s<2^*$, see \cite[Proposition 7]{Troestler}.

The nonlinearities considered in \cite{BB2022,BBS2025} satisfy the
following hypotheses. There are exponents
\[
2<q<p<2^*
\]
such that:

\begin{enumerate}[label=(F\arabic*)]
\item $f:\R\to\R$ is odd, continuous, and
\[
|f(u)|\lesssim 1+|u|^{p-1}, \quad u \in \R;
\]
\item $f(u)=o(|u|)$ as $u\to0$;
\item
$
\frac{F(u)}{|u|^q}\to+\infty
\quad\text{as } |u|\to\infty, \mbox{ where } F(u)=\int_0^u f(s)\,ds, \mbox{ and } 
F(u)\ge0;
$
\item the map
\[
u\mapsto \frac{f(u)}{|u|^{q-1}}
\]
is nondecreasing on $(-\infty,0)$ and on $(0,\infty)$;
\item there exists $\rho>0$ such that for $|u|\ge\rho$,
\[
|u|^{p-1}\lesssim |f(u)|\lesssim |u|^{p-1}.
\]
\end{enumerate}

For $g$ one assumes:
\begin{enumerate}[label=(G\arabic*)]
\item $g:\R\to\R$ is odd, continuous, and 
\[
|g(u)|\lesssim 1+|u|^{q-1}, \quad u \in \R;
\]
\item $g(u)=o(|u|)$ as $u\to0$;
\item the map
\[
u\mapsto \frac{g(u)}{|u|^{q-1}}
\]
is nonincreasing on $(-\infty,0)$ and on $(0,\infty)$, and
\[
g(u)u\ge0.
\]
\end{enumerate}
These assumptions imply
\begin{equation}
0\le qF(u)\le f(u)u,
\end{equation}
and
\begin{equation}
0\le g(u)u\le qG(u).
\end{equation}
Furthermore, for every $\eps>0$ there are constants
$C_{F,\eps},C_{G,\eps}>0$ such that
\begin{equation}
F(u)\ge C_{F,\eps}|u|^q-\eps |u|^2,
\qquad
G(u)\le C_{G,\eps}|u|^q+\eps |u|^2.
\end{equation}
The main example is given by pure powers, see \eqref{eq:pure-power-problem}.

The energy functional associated with \eqref{eq:model-basic} is
\begin{equation}\label{eq:J-general}
J(u)=\frac12\|u^+\|^2-\frac12\|u^-\|^2
-\int_{\R^N}F(u)\,dx+\lambda\int_{\R^N}G(u)\,dx.
\end{equation}
Under (F1)--(F2) and (G1)--(G2), $J\in C^1(X,\R)$, and its critical
points are precisely the weak solutions of \eqref{eq:model-basic}.
For the functional \eqref{eq:J-general} one may write
\[
J(u)=\frac12\|u^+\|^2-\frac12\|u^-\|^2-I(u),
\]
where
\[
I(u)=\int_{\R^N}F(u)\,dx-\lambda\int_{\R^N}G(u)\,dx.
\]
In the classical strongly indefinite theory, one often assumes
$I(u)\ge0$. This is a powerful technical assumption because it yields
good upper semicontinuity properties with respect to the product topology
\[
u_n^+\to u^+,\qquad u_n^-\weak u^-.
\]
For competing nonlinearities, however, the term $\lambda G$ may dominate
locally, especially near the origin, so $I$ need not be nonnegative. Thus
one cannot directly apply the classical Kryszewski-Szulkin linking theorem
or the Nehari-Pankov method.

This obstruction is the central point of \cite{BB2022}. The authors
replace the missing $\tau$-upper semicontinuity by a stronger geometric
condition in a generalized linking theorem. This produces a Cerami
sequence which is bounded away from zero. A careful boundedness argument
and a concentration-compactness procedure then yield a nontrivial critical
point.

The paper is organized as follows. Sections 2 and 3 recall the abstract critical point theory and its applications to existence and multiplicity. Section 4 contains the new ground state result for pure-power nonlinearities.

\section{A generalized linking theorem and the existence result}

Let $X=X^+\oplus X^-$ be a real Hilbert space. On $X$ one introduces the
$\tau$-topology generated by the norm
\[
\vertiii{u}=\max\left\{
\|u^+\|,
\sum_{k=1}^\infty 2^{-k-1}|(u^-,e_k)|
\right\},
\]
where $(e_k)$ is a complete orthonormal sequence in $X^-$. For bounded
sequences one has
\[
u_n\xrightarrow{\tau}u
\quad\Longleftrightarrow\quad
u_n^+\to u^+ \text{ in }X^+,\ u_n^-\weak u^- \text{ in }X^-.
\]

For $u\in X\setminus X^-$ and $R>r>0$ define
\[
S_r^+=\{v\in X^+:\|v\|=r\}
\]
and
\[
M(u)=\{tu+v^-:t\ge0,\ v^-\in X^-,\ \|tu+v^-\|\le R\}.
\]
This is the linking cylinder generated by the direction $u$ and the space $X^-$. The main abstract theorem of \cite{BB2022} may be summarized as follows.

\begin{theorem}[Generalized linking, {\cite[Theorem 2.1]{BB2022}}]
Let $J\in C^1(X,\R)$, $J(0)=0$, and assume that $J'$ is sequentially
weak-to-weak$^*$ continuous. Suppose that there exist $\delta>0$, $r>0$,
and a nonempty set $\cP\subset X\setminus X^-$ such that for every
$u\in\cP$ one can choose $R=R(u)>r$ satisfying
\[
\inf_{S_r^+}J>
\max\left\{
\sup_{\partial M(u)}J,\,
\sup_{\vertiii{v}\le\delta}J(v)
\right\}.
\]
Then there exists a Cerami sequence $(u_n)\subset X$ such that
\[
\sup_n J(u_n)\le c,\qquad
(1+\|u_n\|)J'(u_n)\to0,
\qquad
\inf_n \vertiii{u_n}\ge\frac{\delta}{2},
\]
where
\[
c\ge \inf_{S_r^+}J>0.
\]
\end{theorem}

The level $c>0$ is a minimax level introduced in \cite{BB2022}. We omit its precise definition here and refer the reader to \cite{BB2022}.

The importance of this theorem lies in the fact that it does not require
$\tau$-upper semicontinuity of the functional. Instead, the geometric
condition contains the additional inequality
\[
\inf_{S_r^+}J>\sup_{\vertiii{v}\le\delta}J(v),
\]
which prevents the minimax sequence from collapsing to zero in the
$\tau$-topology.

For the functional \eqref{eq:J-general}, one may take
$\cP=X^+\setminus\{0\}$ and verify the following three parts of the
linking geometry.

First, there exists $r>0$ such that
\[
\inf_{S_r^+}J>0.
\]
This follows from the fact that, on a sufficiently small sphere in $X^+$,
the positive quadratic part dominates the nonlinear terms, since
$f(u)=o(u)$ and $g(u)=o(u)$ as $u\to0$.

Second, for every $u\in X^+\setminus\{0\}$ there exists $R(u)>r$ such
that
\[
\sup_{\partial M(u)}J\le0.
\]
Here the crucial point is that $F$ grows faster than $|u|^q$ at infinity,
whereas $\lambda>0$ is sufficiently small. The focusing part forces the
energy to go to $-\infty$ on the boundary of the linking cylinder.

Third, for some $\delta>0$,
\[
\sup_{\vertiii{u}\le\delta}J(u)<\inf_{S_r^+}J.
\]
This condition is directly connected with the construction in
\cite{BB2022}, originally introduced in \cite{ChenWang}. It substitutes the missing upper semicontinuity.

Once the Cerami sequence bounded away from zero is obtained, one proves
its boundedness. The key boundedness statement is the following.

\begin{theorem}[Boundedness of Cerami sequences, {\cite[Proposition 5.2]{BB2022}}]
Assume that $\lambda>0$ and the parameter $\rho>0$ from (F5) are
sufficiently small. If $(u_n)\subset X$ satisfies
\[
J(u_n)\le\beta,\qquad
(1+\|u_n\|)J'(u_n)\to0,
\]
then $(u_n)$ is bounded in $X$.
\end{theorem}

The proof splits the domain into the regions $\{|u_n|<\rho\}$ and
$\{|u_n|\ge\rho\}$. On the first region, the nonlinear terms are small
with respect to the linear part. On the second region, condition (F5)
allows one to control $g$ by $f$, provided that $\lambda$ is sufficiently
small.

Having boundedness, one uses a concentration-compactness principle. If the sequence vanishes, then
\[
|u_n|_s\to0,\qquad 2<s<2^*,
\]
which contradicts the fact that the sequence is bounded away from zero in
the $\tau$-topology. Hence, after translations in $\Z^{N}$, the sequence
has a nontrivial weak limit. By weak-to-weak$^*$ continuity of $J'$, this
weak limit is a nontrivial critical point.

Thus, we obtain the following existence result.

\begin{theorem}[Existence of a nontrivial solution, {\cite[Theorem 1.2]{BB2022}}]\label{thm:BB-existence}
Assume (\ref{ass:V}), (F1)--(F5), and (G1)--(G3). If $\lambda>0$ and the
parameter $\rho>0$ from (F5) are sufficiently small, then
\eqref{eq:model-basic} has a nontrivial weak solution.
\end{theorem}

\section{Multiplicity of solutions}

The paper \cite{BBS2025} develops an abstract multiplicity theory for
strongly indefinite functionals with a sign-changing nonlinear part. The
main point is that the lack of compactness caused by translations can be
handled by combining a profile decomposition with a suitable deformation
argument. In contrast to more classical strongly indefinite theories, the
functional is not required to be upper semicontinuous with respect to the
$\tau$-topology.

We first introduce the notation needed to formulate the abstract result.
Let $X$ be a real separable Hilbert space endowed with an orthogonal
decomposition
\[
X=X^+\oplus X^-.
\]
We assume that a group $\cG$ acts unitarily on $X$. Thus, there is a group
homomorphism
\[
\cG\ni \gamma\longmapsto T_\gamma\in \mathrm{GL}(X)
\]
such that every $T_\gamma$ is unitary. For simplicity, we write $\gamma u$ instead of
$T_\gamma u$.

A functional $J:X\to\R$ is called $\cG$-invariant if
\[
J(\gamma u)=J(u),
\qquad
\gamma\in\cG,\quad u\in X.
\]
If $J$ is $\cG$-invariant, then its gradient is $\cG$-equivariant, namely
\[
\nabla J(\gamma u)=\gamma\nabla J(u),
\qquad
\gamma\in\cG,\quad u\in X.
\]
Consequently, the set of critical points
\[
\crit(J):=\{u\in X:J'(u)=0\}
\]
is $\cG$-invariant. For $u\in X$, we denote by
\[
\mathcal{O}(u):=\{\gamma u:\gamma\in\cG\}
\]
the orbit of $u$ under the action of $\cG$. Two critical points $u,v\in
\crit(J)$ are called \emph{geometrically distinct} if
\[
\mathcal{O}(u)\neq \mathcal{O}(v).
\]

In the periodic Schr\"odinger setting, we take
\[
\cG=\Z^N,
\]
acting on $X=H^1(\R^N)$ by integer translations,
\[
(k*u)(x):=u(x-k),
\qquad
k\in\Z^N.
\]
We identify the action of $k\in\cG$ with the operator $u\mapsto k*u$ and
write $ku:=k*u$ when using the abstract notation.
Since $V$ is $\Z^N$-periodic, the operator $-\Delta+V$ commutes with this
action. Hence the spectral subspaces $X^+$ and $X^-$ are
$\Z^N$-invariant, and the functional $J$ is $\Z^N$-invariant. Moreover,
since $f$ and $g$ are odd, their primitives $F$ and $G$ are even, and
therefore $J$ is even:
\[
J(-u)=J(u),
\qquad
u\in X.
\]

\subsection{Dislocation spaces and \texorpdfstring{$\cG$}{G}-weak convergence}

The relevant compactness notion is stronger than ordinary weak convergence.
For a sequence $(u_n)\subset X$ and $u\in X$, we write
\[
u_n\rightharpoonup_{\cG} u
\]
if
\[
\lim_{n\to\infty}
\sup_{\gamma\in\cG}
\langle u_n-u,\gamma\varphi\rangle
=0
\]
for every $\varphi\in X$. Thus, $u_n\rightharpoonup_{\cG} u$ means that no
nontrivial weak component can be recovered from $u_n-u$ after applying an
arbitrary element of the group $\cG$.

We also use weak convergence of operators. A sequence $(\gamma_n)\subset\cG$
satisfies
\[
\gamma_n\rightharpoonup 0
\]
if
\[
\gamma_n\varphi\rightharpoonup 0
\qquad\text{in }X
\]
for every $\varphi\in X$.

\begin{definition}
The pair $(X,\cG)$ is called a \emph{dislocation space} if, for every
sequence $(u_n)\subset X$ and every sequence $(\gamma_n)\subset\cG$,
\[
\gamma_n\not\rightharpoonup 0,
\qquad
u_n\rightharpoonup 0
\]
imply, up to a subsequence,
\[
\gamma_nu_n\rightharpoonup 0.
\]
\end{definition}

This condition expresses the fact that weakly convergent profiles cannot
be changed into new nonzero weak limits by a bounded sequence of group
actions. It is precisely the structural assumption needed for abstract
profile decompositions.

For the translation action of $\Z^N$ on $H^1(\R^N)$, one has
\[
k_n\rightharpoonup 0
\quad\Longleftrightarrow\quad
|k_n|\to\infty.
\]
Moreover, if $(u_n)$ is bounded in $H^1(\R^N)$, then
\[
u_n\rightharpoonup_{\Z^N}0
\quad\Longleftrightarrow\quad
|u_n|_s\to0
\]
for every
\[
2<s<2^*.
\]
Hence, in the periodic Schr\"odinger case, $\cG$-weak convergence to zero is
equivalent to Lions-type vanishing in all subcritical Lebesgue spaces.

\subsection{Profile decompositions}

A basic consequence of the dislocation-space structure is that every
bounded sequence can be decomposed into a family of asymptotically
separated profiles.

\begin{theorem}[Profile decomposition, {\cite[Theorem 3.1]{TintarevFieseler}}]
Let $(X,\cG)$ be a dislocation space and let $(u_n)\subset X$ be bounded.
Then, after passing to a subsequence, there exist
\[
K\in\{0,1,2,\ldots\}\cup\{\infty\},
\]
nonzero elements
\[
w^j\in X\setminus\{0\},
\qquad
1\leq j\leq K,
\]
and sequences
\[
(\gamma_n^j)\subset\cG,
\qquad
1\leq j\leq K,
\]
such that, if $K\geq1$, then $\gamma_n^1=\mathrm{Id}$ and
\[
(\gamma_n^j)^{-1}u_n\rightharpoonup w^j
\]
for every $j$. Moreover, for $j\neq\ell$,
\[
\gamma_n^j(\gamma_n^\ell)^{-1}\rightharpoonup0,
\]
and
\[
\sum_{j=1}^K\|w^j\|^2
\leq
\limsup_{n\to\infty}\|u_n\|^2.
\]
Finally,
\[
u_n-\sum_{j=1}^K \gamma_n^j w^j
\rightharpoonup_{\cG}0.
\]
\end{theorem}

The condition
\[
\gamma_n^j(\gamma_n^\ell)^{-1}\rightharpoonup0
\]
means that the profiles are asymptotically separated by the group action.
In the periodic case, this amounts to saying that the corresponding
translation parameters diverge from one another. Thus, distinct profiles
concentrate in mutually diverging lattice cells.

For Palais--Smale sequences, the profiles inherit the critical-point
property. Indeed, suppose that
\[
J(u)
=
\frac12\|u^+\|^2
-
\frac12\|u^-\|^2
-
I(u)
\]
is $\cG$-invariant, $I\in C^1(X,\R)$, and $J'$ is sequentially
weak-to-weak$^*$ continuous. If $(u_n)$ is a bounded
Palais--Smale sequence, then every profile $w^j$ in the preceding
decomposition satisfies
\[
J'(w^j)=0.
\]
Furthermore, if
\[
\eta:=
\inf\{\|u\|:u\in\crit(J)\setminus\{0\}\}>0,
\]
then only finitely many profiles may occur, since
\[
K
\leq
\frac{\limsup_{n\to\infty}\|u_n\|^2}{\eta^2}.
\]

The stronger conclusion needed in the multiplicity argument follows from
the following uniform continuity property.

\begin{definition}
We say that $I'$ satisfies the \emph{$\cG$-weak continuity property},
abbreviated by \emph{(GWC)}, if, whenever $(v_n)\subset X$ is bounded and
\[
\varphi_n\rightharpoonup_{\cG}0,
\]
one has
\[
I'(v_n)(\varphi_n)\to0.
\]
\end{definition}

If, in addition,
\[
\inf\{\|u\|:u\in\crit(J)\setminus\{0\}\}>0,
\]
and $I'$ satisfies \emph{(GWC)}, then the $\cG$-weakly vanishing
remainder in the profile decomposition converges strongly to zero. Thus,
for a bounded Palais--Smale sequence one obtains
\[
u_n-\sum_{j=1}^K \gamma_n^j w^j\to0
\qquad\text{in }X.
\]
This strong profile decomposition is the main compactness ingredient in
the multiplicity theory.

\subsection{The discreteness property}

The multiplicity proof requires a strengthening of the fact that individual
orbits are discrete. Let $A\subset X$ be finite and let
$\ell\in\N$. Define
\[
[A,\ell]
:=
\left\{
\sum_{i=1}^j \gamma_i u_i:
1\leq j\leq\ell,\quad
\gamma_i\in\cG,\quad
u_i\in A
\right\}.
\]
Thus, $[A,\ell]$ consists of all superpositions of at most $\ell$ group
translates of elements of $A$.

\begin{definition}
A dislocation space $(X,\cG)$ is said to have the \emph{discreteness
property} if, for every finite set $A\subset X$ and every $\ell\in\N$,
\[
\inf
\left\{
\|u-v\|:
u,v\in[A,\ell],\quad u\neq v
\right\}
>0.
\]
\end{definition}

In particular, taking $\ell=1$ shows that the orbit
\[
\mathcal{O}(u)=\{\gamma u:\gamma\in\cG\}
\]
of every $u\in X$ is a discrete subset of $X$. The full discreteness
property is substantially stronger. It also ensures uniform separation of
all finite superpositions of translated profiles. This is important because
a Palais--Smale sequence may split into several spatially separated
critical profiles.

For the periodic problem,
\[
\bigl(H^1(\R^N),\Z^N\bigr)
\]
is a dislocation space with the discreteness property; see
\cite[Example~2.13]{BBS2025}. Therefore, the
abstract result applies to the translation action induced by the periodic
potential.

\subsection{Abstract multiplicity theorem}

We now state the abstract result from \cite{BBS2025} in a form adapted to
the present setting. Let
\[
J(u)
=
\frac12\|u^+\|^2
-
\frac12\|u^-\|^2
-
I(u),
\qquad
u=u^++u^-\in X^+\oplus X^-.
\]
For $r>0$, put
\[
S_r^+
:=
\{u\in X^+:\|u\|=r\}.
\]

\begin{theorem}[Multiplicity of critical orbits, {\cite[Theorem 3.2]{BBS2025}}]\label{thm:multiplicity-abstract}
Suppose that the following conditions hold:
\begin{enumerate}[label=\textup{(A\arabic*)}]
\item
$(X,\cG)$ is a dislocation space with the discreteness property, and
$X=X^+\oplus X^-$ is an orthogonal decomposition into $\cG$-invariant
subspaces, with $X^+$ infinite-dimensional;

\item
$I\in C^1(X,\R)$, $I(0)=0$, $I'$ is sequentially weak-to-weak$^*$
continuous, and $I'$ satisfies (GWC);

\item
$J$ is even and $\cG$-invariant;

\item
$J$ possesses at least one nontrivial critical point;

\item
there exist $r>0$ and $r_0>0$ such that
\[
\inf_{S_r^+}J
>
\sup\{J(u):u\in X,\ \|u^+\|<r_0\};
\]

\item
whenever
\[
\|u_n^-\|\to\infty
\]
and $(u_n^+)$ is bounded in $X^+$, one has
\[
J(u_n)\to-\infty;
\]

\item
$J$ is bounded from above on bounded subsets of $X$;

\item
if $(u_n^+)$ is contained in a finite-dimensional subspace of $X^+$ and
\[
\|u_n^+\|\to\infty,
\]
then
\[
\frac{I(u_n^+)}{\|u_n^+\|^2}\to\infty;
\]

\item
for every $\beta\in\R$, there exists $M_\beta>0$ such that every sequence
$(u_n)\subset X$ satisfying
\[
\limsup_{n\to\infty}J(u_n)\leq\beta,
\qquad
J'(u_n)\to0,
\]
also satisfies
\[
\limsup_{n\to\infty}\|u_n\|\leq M_\beta.
\]
\end{enumerate}
Then $J$ has infinitely many geometrically distinct critical points.
\end{theorem}

The proof is based on a Benci-type pseudoindex construction. Assuming that
there are only finitely many critical orbits, one chooses a finite set
$A\subset\crit(J)$ containing representatives of these orbits. The profile
decomposition and property \emph{(GWC)} imply that Palais--Smale sequences
at a fixed level are asymptotically close to sets of the form
\[
[A,\ell]
\]
for a suitable $\ell\in\N$. The discreteness property gives a positive
uniform distance between distinct points of $[A,\ell]$. This makes it
possible to construct odd admissible deformations avoiding neighborhoods
of all critical profiles. The resulting pseudoindex argument yields a
contradiction with the assumption that only finitely many critical orbits
exist.

\subsection{Application to the periodic Schr\"odinger equation}

We now return to the functional
\[
J(u)
=
\frac12\|u^+\|^2
-
\frac12\|u^-\|^2
-
\int_{\R^N}F(u)\,dx
+
\lambda\int_{\R^N}G(u)\,dx.
\]
Under assumption \textup{(V)}, the spectral splitting
\[
X=H^1(\R^N)=X^+\oplus X^-
\]
is invariant under the action of $\Z^N$. Moreover, the periodicity of $V$
implies that $J$ is $\Z^N$-invariant, while the oddness of $f$ and $g$
implies that $J$ is even.

The nonlinear part
\[
I(u)
=
\int_{\R^N}F(u)\,dx
-
\lambda\int_{\R^N}G(u)\,dx
\]
satisfies the property \emph{(GWC)}. Indeed, if $(v_n)$ is bounded in
$H^1(\R^N)$ and
\[
\varphi_n\rightharpoonup_{\Z^N}0,
\]
then
\[
|\varphi_n|_p\to0,
\qquad
|\varphi_n|_q\to0.
\]
The growth assumptions on $f$ and $g$ then imply
\[
I'(v_n)(\varphi_n)\to0.
\]

The remaining assumptions of Theorem~\ref{thm:multiplicity-abstract} follow from the linking geometry, the
anti-coercivity of $J$ in $X^-$, and the
boundedness result for Palais--Smale sequences established in
\cite{BBS2025}. Consequently, one obtains the following multiplicity
result.

\begin{theorem}[Multiplicity for the periodic Schr\"odinger equation, {\cite[Theorem 4.7]{BBS2025}}]
Assume (\ref{ass:V}), (F1)--(F5), and (G1)--(G3). If
$\lambda>0$ and the parameter $\rho>0$ from (F5) are sufficiently
small, then
\[
-\Delta u+V(x)u=f(u)-\lambda g(u)
\]
has infinitely many pairs $\pm u$ of geometrically distinct weak
solutions.
\end{theorem}

\section{A new result: ground states in the pure-power case}

We now turn to the new part of the paper. We consider the pure-power
nonlinearities
\[
f(u)=|u|^{p-2}u,\qquad
g(u)=|u|^{q-2}u,
\qquad 2<q<p<2^*.
\]
These nonlinearities satisfy (F1)--(F5) and (G1)--(G3); in particular,
(F5) holds for every choice of $\rho>0$. The energy functional becomes
\begin{equation}\label{eq:J-power}
J(u)=\frac12\|u^+\|^2-\frac12\|u^-\|^2
-\frac1p\int_{\R^N}|u|^p\,dx
+\frac{\lambda}{q}\int_{\R^N}|u|^q\,dx.
\end{equation}
Its derivative is
\begin{equation}
J'(u)[\varphi]
=
\langle u^+,\varphi^+\rangle-\langle u^-,\varphi^-\rangle
-\int_{\R^N}|u|^{p-2}u\varphi\,dx
+\lambda\int_{\R^N}|u|^{q-2}u\varphi\,dx.
\end{equation}

Let
\[
\cK:=\{u\in X\setminus\{0\}:J'(u)=0\}
\]
be the set of all nontrivial critical points.

\begin{definition}
A point $u\in\cK$ is called a ground state if
\[
J(u)=\inf_{v\in\cK}J(v).
\]
\end{definition}

We prove the following theorem.

\begin{theorem}[Existence of a ground state]\label{thm:ground-state}
Assume (\ref{ass:V}). Let
\[
2<q<p<2^*
\]
and
\[
f(u)=|u|^{p-2}u,\qquad g(u)=|u|^{q-2}u.
\]
Then there exists $\lambda_*>0$ such that for every
$\lambda\in(0,\lambda_*)$ the functional \eqref{eq:J-power} possesses a
ground state. Equivalently, the corresponding strongly indefinite
Schrödinger equation admits a nontrivial solution of least energy among
all nontrivial solutions.
\end{theorem}

We start with a simple interpolation-type estimate.

\begin{lemma}\label{lem:power-two}
Let $2<q<p$. For every $\eps>0$ there exists $C_\eps>0$ such that
\[
|t|^q\le \eps |t|^p+C_\eps |t|^2,\qquad t\in\R.
\]
Consequently,
\[
|u|_q^q\le \eps |u|_p^p+C_\eps |u|_2^2
\]
for every $u\in L^p(\R^N)\cap L^2(\R^N)$.
\end{lemma}

\begin{proof}
Since $2<q<p$, the quotient
\[
\frac{|t|^q}{|t|^p+|t|^2}
\]
tends to zero both as $t\to0$ and as $|t|\to\infty$. A standard splitting
into the regions of small and large $|t|$ gives the pointwise inequality.
Integration gives the asserted estimate for functions.
\end{proof}

Since the projections onto $X^\pm$ are continuous in $L^s(\R^N)$, for
every $s\in[2,2^*)$ there exists $D_s>0$ such that
\begin{equation}\label{eq:plus-minus}
|u^+-u^-|_s\le D_s |u|_s,\qquad u\in X.
\end{equation}

\begin{lemma}\label{lem:away-zero-critical}
There exists $\eta_0>0$, independent of $\lambda\in(0,1]$, such that every
$u\in\cK$ satisfies
\[
\|u\|\ge \eta_0.
\]
\end{lemma}

\begin{proof}
Let $u\in\cK$. Testing the equation $J'(u)=0$ with
$\varphi=u^+-u^-$, we obtain
\[
\|u\|^2
=
\int_{\R^N}|u|^{p-2}u(u^+-u^-)\,dx
-\lambda\int_{\R^N}|u|^{q-2}u(u^+-u^-)\,dx.
\]
Hence, by Hölder's inequality and \eqref{eq:plus-minus},
\[
\|u\|^2
\le D_p|u|_p^p+\lambda D_q|u|_q^q.
\]
Using the continuous embeddings $X\hookrightarrow L^p(\R^N)\cap
L^q(\R^N)$ and $\lambda\le1$, we get
\[
\|u\|^2\le C(\|u\|^p+\|u\|^q),
\]
for some $C>0$. If there were a sequence $u_n\in\cK$ such that $\|u_n\|\to0$, then after
dividing by $\|u_n\|^2$ we would obtain
\[
1\le C(\|u_n\|^{p-2}+\|u_n\|^{q-2})\to0,
\]
a contradiction.
\end{proof}

The following estimate is the key point of the proof.

\begin{proposition}\label{prop:positive-ps}
There exists $\lambda_1>0$ such that, for every
$\lambda\in(0,\lambda_1)$ and every $\eta>0$, there is a constant
$c_\eta>0$ with the following property. If $(u_n)\subset X$ is bounded,
$J'(u_n)\to0$ in $X^*$, and
\[
\liminf_{n\to\infty}\|u_n\|\ge\eta,
\]
then
\[
\liminf_{n\to\infty}J(u_n)\ge c_\eta>0.
\]
\end{proposition}

\begin{proof}
Let $(u_n)$ be a bounded sequence such that
\[
J'(u_n)\to0,
\qquad
\liminf_{n\to\infty}\|u_n\|\ge\eta>0.
\]
Testing $J'(u_n)\to0$ with $u_n^+-u_n^-$ gives
\[
\|u_n\|^2
\le D_p|u_n|_p^p+\lambda D_q|u_n|_q^q+o(1).
\]
By Lemma~\ref{lem:power-two}, with $\eps=1$, and by \eqref{eq:mu0},
\[
|u_n|_q^q\le |u_n|_p^p+C|u_n|_2^2
\le |u_n|_p^p+C\mu_0^{-1}\|u_n\|^2.
\]
Hence
\[
\|u_n\|^2
\le C|u_n|_p^p+\lambda C\|u_n\|^2+o(1).
\]
After fixing $\lambda_1>0$ sufficiently small, the last term can be
absorbed into the left-hand side. Thus, for a constant $C>0$ independent of
$n$ and $\lambda\in(0,\lambda_1)$,
\begin{equation}\label{eq:norm-by-Lp}
\|u_n\|^2\le C|u_n|_p^p+o(1).
\end{equation}

Set
\[
\alpha:=\frac12-\frac1p>0,
\qquad
\beta:=\frac12-\frac1q>0.
\]
For every $u\in X$,
\[
J(u)-\frac12J'(u)[u]
=
\alpha |u|_p^p-\lambda\beta |u|_q^q.
\]
Using Lemma~\ref{lem:power-two} with a fixed sufficiently small $\eps>0$,
then \eqref{eq:mu0} and \eqref{eq:norm-by-Lp}, we obtain
\[
|u_n|_q^q\le C|u_n|_p^p+o(1).
\]
After decreasing $\lambda_1$ if necessary,
\[
J(u_n)-\frac12J'(u_n)[u_n]
\ge \frac{\alpha}{2}|u_n|_p^p+o(1).
\]
Since $(u_n)$ is bounded and $J'(u_n)\to0$, we have
$J'(u_n)[u_n]=o(1)$. Moreover, \eqref{eq:norm-by-Lp} gives
\[
\liminf_{n\to\infty}|u_n|_p^p\ge \frac{\eta^2}{C}.
\]
Consequently,
\[
\liminf_{n\to\infty}J(u_n)
\ge \frac{\alpha\eta^2}{2C}=:c_\eta>0.
\]
\end{proof}

By Theorem \ref{thm:BB-existence}, there is $\lambda_{\mathrm{ex}}>0$
such that $\cK\neq\emptyset$ for every
$\lambda\in(0,\lambda_{\mathrm{ex}})$; here we use the fact that, in
the pure-power case, condition (F5) holds for every $\rho>0$. Set
\[
\lambda_*:=\min\{1,\lambda_1,\lambda_{\mathrm{ex}}\},
\]
and fix $\lambda\in(0,\lambda_*)$ for the remainder of this section.
Define
\[
m:=\inf_{u\in\cK}J(u).
\]

\begin{corollary}\label{cor:critical-positive}
For $\lambda>0$ sufficiently small there exists $c_0>0$ such that
\[
J(u)\ge c_0,
\qquad u\in\cK.
\]
\end{corollary}

\begin{proof}
Apply Proposition~\ref{prop:positive-ps} to the constant sequence
$u_n\equiv u\in\cK$ with $\eta=\eta_0$ from
Lemma~\ref{lem:away-zero-critical}. The resulting constant
$c_0=c_{\eta_0}$ is independent of $u$.
\end{proof}

By Corollary~\ref{cor:critical-positive},
\[
0<c_0\le m<\infty.
\]
Choose a minimizing sequence $(u_n)\subset\cK$ such that
\[
J(u_n)\to m.
\]

\begin{lemma}\label{lem:min-bounded}
The sequence $(u_n)$ is bounded in $X$.
\end{lemma}

\begin{proof}
Since $u_n\in\cK$, we have $J'(u_n)=0$. Repeating the first part of the
proof of Proposition~\ref{prop:positive-ps}, now without the $o(1)$ term,
yields
\[
\|u_n\|^2\le C|u_n|_p^p.
\]
Likewise, the second part of that proof gives
\[
J(u_n)
=J(u_n)-\frac12J'(u_n)[u_n]
\ge\frac{\alpha}{2}|u_n|_p^p,
\]
where $\alpha=\frac12-\frac1p>0$. Since $J(u_n)\to m$, the sequence
$(|u_n|_p)$ is bounded, and the first estimate implies that $(u_n)$ is
bounded in $X$.
\end{proof}

\begin{lemma}\label{lem:no-vanishing}
The sequence $(u_n)$ does not vanish. More precisely, there exist
$R>0$, $\delta>0$, and a sequence $(z_n)\subset\Z^{N}$ such that, up to
a subsequence,
\[
\int_{B(z_n,R)}|u_n|^2\,dx\ge\delta.
\]
\end{lemma}

\begin{proof}
By Lemma~\ref{lem:min-bounded}, the sequence $(u_n)$ is bounded. Suppose by contradiction that $(u_n)$ vanishes, namely for every $R>0$,
\[
\sup_{z\in\R^{N}}\int_{B(z,R)}|u_n|^2\,dx\to0.
\]
By Lions' concentration--compactness argument (\cite{Lions}),
\[
|u_n|_p\to0,\qquad |u_n|_q\to0.
\]
Since $J'(u_n)=0$, testing the equation with $u_n^+-u_n^-$ gives
\[
\|u_n\|^2
\le D_p|u_n|_p^p+\lambda D_q|u_n|_q^q\to0.
\]
This contradicts Lemma~\ref{lem:away-zero-critical}. Hence there are
$R>0$, $\delta>0$, and points $y_n\in\R^N$ such that
\[
\int_{B(y_n,R)}|u_n|^2\,dx\ge\delta.
\]
Choosing $z_n\in\Z^N$ with $|z_n-y_n|\le \sqrt N/2$ and enlarging $R$
if necessary yields the asserted conclusion.
\end{proof}

By Lemma~\ref{lem:no-vanishing}, after translations we may assume that
\[
v_n(x):=u_n(x+z_n)
\]
satisfies
\[
\int_{B(0,R)}|v_n|^2\,dx\ge\delta.
\]
The functional is invariant under $\Z^{N}$-translations, hence
\[
J(v_n)=J(u_n),\qquad J'(v_n)=0.
\]
The sequence $(v_n)$ is bounded, so, up to a subsequence,
\[
v_n\weak u_*\quad\text{in }X.
\]
By the local compactness of the embedding into $L^2_{\mathrm{loc}}(\R^N)$ and
the lower bound on the local $L^2$-mass,
\[
u_*\ne0.
\]
Moreover, the weak-to-weak$^*$ continuity of $J'$ gives
\[
J'(u_*)=0.
\]
Thus $u_*\in\cK$, and therefore
\[
J(u_*)\ge m.
\]

To prove the reverse inequality, we use a Brezis--Lieb splitting.

\begin{lemma}[Splitting lemma]\label{lem:splitting}
Let
\[
w_n:=v_n-u_*.
\]
Then, up to a subsequence,
\[
J(v_n)=J(u_*)+J(w_n)+o(1),
\]
and
\[
J'(w_n)\to0\quad\text{in }X^*.
\]
\end{lemma}

\begin{proof}
Since $v_n\weak u_*$ in $X$, the orthogonality of the splitting gives
\[
\|v_n^\pm\|^2=\|u_*^\pm\|^2+\|w_n^\pm\|^2+o(1).
\]
Moreover, after passing to a subsequence, $v_n(x)\to u_*(x)$ almost
everywhere, and the sequence is bounded in both $L^p(\R^N)$ and $L^q(\R^N)$. By the Brezis--Lieb lemma \cite{BrezisLieb},
\[
|v_n|_p^p=|u_*|_p^p+|w_n|_p^p+o(1),
\qquad
|v_n|_q^q=|u_*|_q^q+|w_n|_q^q+o(1).
\]
Substitution into the definition of $J$ yields the energy splitting.

For the derivative, the quadratic part splits linearly, while the
power-type nonlinearities satisfy the dual Brezis--Lieb splitting
\[
|v_n|^{r-2}v_n-|w_n|^{r-2}w_n-|u_*|^{r-2}u_*
\to0
\quad\text{in }L^{r/(r-1)}(\R^N)
\]
for $r\in\{p,q\}$. Since $J'(v_n)=0$ and $J'(u_*)=0$, we conclude that
\[
J'(w_n)\to0.
\]
\end{proof}

\begin{proof}[Proof of Theorem~\ref{thm:ground-state}]
Let $\lambda\in(0,\lambda_*)$. The preceding construction provides a
sequence $(v_n)\subset\cK$ such
that
\[
J(v_n)\to m,\qquad v_n\weak u_*\ne0.
\]
Moreover $u_*\in\cK$ and $J(u_*)\ge m$.

Let $w_n=v_n-u_*$. By Lemma~\ref{lem:splitting},
\[
J(v_n)=J(u_*)+J(w_n)+o(1).
\]
Hence
\[
J(w_n)=m-J(u_*)+o(1).
\]
Since $J(u_*)\ge m$, we get
\[
\limsup_{n\to\infty}J(w_n)\le0.
\]
At the same time,
\[
J'(w_n)\to0.
\]

We claim that $w_n\to0$ strongly in $X$. Otherwise, after passing to a
subsequence, there is $\eta>0$ such that
\[
\|w_n\|\ge\eta
\qquad\text{for every }n.
\]
The sequence $(w_n)$ is bounded and $J'(w_n)\to0$, so Proposition~\ref{prop:positive-ps} gives
\[
\liminf_{n\to\infty}J(w_n)>0.
\]
This contradicts $\limsup_{n\to\infty} J(w_n)\le0$. Therefore $w_n\to0$ in $X$.

Consequently,
\[
m=\lim_{n\to\infty}J(v_n)=J(u_*).
\]
Since $u_*\in\cK$, the point $u_*$ is a ground state.
\end{proof}

\section*{Acknowledgments}
Bartosz Bieganowski was partly supported by the National Science Centre, Poland (Grant No. 2022/47/D/ST1/00487).

\end{document}